\documentclass[10pt,a4paper]{article}
\usepackage{ifthen}
\newboolean{MareksFormat}
\InputIfFileExists{MareksSwitch}{}{}
\ifthenelse{\boolean{MareksFormat}}{
 \usepackage{geometry}
\geometry{papersize={158mm,196mm},top=2.5mm,bottom=2.5mm,left=1.5mm,right=1.5mm}
}{\usepackage{a4wide}}

\usepackage{latexsym}
\usepackage{verbatim}
\usepackage{amsmath}
\usepackage{amsfonts}
\usepackage{amssymb}
\usepackage[utf8]{inputenc}
\usepackage[dvipsnames]{xcolor}
\usepackage[matrix,arrow,cmtip,rotate,curve,arc]{xy}
\usepackage{enumitem}
\usepackage{rotating}
\usepackage{amsmath}
\usepackage{amsfonts}
\usepackage{amssymb}
\usepackage{amscd}
\usepackage{amsthm}
\usepackage{stmaryrd}
\usepackage[matrix,arrow,cmtip,rotate,curve,arc]{xy}
\SelectTips{cm}{10}
\usepackage{graphics}
\usepackage{graphicx}
\usepackage{fancyhdr}
\usepackage{ushort}
\usepackage{enumitem}
\usepackage{color}
\usepackage{mathtools}
\usepackage{url}
\usepackage{float}
\input xy
\xyoption{all}
\usepackage{soul}
\setstcolor{orange}
\setulcolor{orange}
\setul{}{0.3ex}

\theoremstyle{plain}
\newtheorem{thm}{Theorem}[section]
\newtheorem{cor}[thm]{Corollary}
\newtheorem{lem}[thm]{Lemma}
\newtheorem{prop}[thm]{Proposition}

\theoremstyle{definition}
\newtheorem{Def}[thm]{Definition}

\newtheorem{ex}[thm]{Example}

\newcommand{\bbZ}{\mathbb Z}
\newcommand{\mcA}{\mathcal A}

\newcommand{\col}{\!:\:\!}

\newcommand{\cof}[1][]{\xymatrix@1@C=15pt{{}\ar@{>->}[r]^{#1} & {}}}
\newcommand{\fib}[1][]{\xymatrix@1@C=15pt{{}\ar@{->>}[r]^{#1} & {}}}
\newcommand{\id}{\mathop\mathrm{id}\nolimits}

\newcommand{\ef}{\mathrm{ef}}

\newcommand{\Ker}{\mathop\mathrm{Ker}\nolimits}
\newcommand{\coker}{\mathop\mathrm{coker}\nolimits}

\newcommand{\op}{\mathsf{op}}
\newcommand{\cofib}{\mathsf{cof}}
\newcommand{\Hom}{\mathsf{Hom}}
\newcounter{sideremark}
\def\Z{\mathbb Z}
\def\Z2{\mathbb Z_2}

\newcommand{\Ch}{\mathsf{Ch}}
\newcommand{\sSet}{\mathsf{sSet}}

\newcommand{\Ab}{\mathsf{Ab}}
\newcommand{\OgsSet}{{\mathsf{sSet}}^{\mathcal{O}_G ^\mathsf{op}}}

\newcommand{\GsSet}{{\mathsf{sSet}^G}}
\newcommand{\Ocat}{\mathcal{O}}

\newlength{\hlp}

\newcommand{\stdsimp}[1]{\Delta^{#1}}
\newcommand{\horn}[2]{%
\mbox{$\xy
<0pt,-\the\fontdimen22\textfont2>;p+<.1em,0em>:
{\ar@{-}(0,0.1);(3,7)},
{\ar@{-}(3,7);(6,0.1)},
{\ar@{-}(3.2,7);(6.2,0.1)},
{\ar@{-}(3.4,7);(6.4,0.1)}
\endxy\;\!{}^{#1}_{#2}$}}

\def\bo{\partial} 

\newcommand{\holan}{\mathsf{hoLan}}

\newcommand{\Ra}{\xRightarrow{\ \ }}
\newcommand{\La}{\xLeftarrow{\ \ }}
\newcommand{\LRa}{\xLeftrightarrow{\ \ \ }}
\newcommand{\ccat}{\mathcal C}

\newcommand{\icat}{\mathcal I}
\newcommand{\jcat}{\mathcal J}

\newcommand{\hocolim}{\mathsf{hocolim} \,}

\newcommand{\sk}{\mathsf{sk}}
\begin{document}

\title{Effective homology for homotopy colimit and cofibrant replacement\footnote{The research of M. F. was supported by Masaryk University project MUNI/A/0821/2013}}
\author{M. Filakovsk\'{y}}
\date{\today}
\maketitle

\abstract{
We extend the notion of simplicial set with effective homology presented in \cite{serg} to diagrams of simplicial sets. Further, for a given finite diagram of simplicial sets $X \colon \icat \to \sSet$ such that each simplicial set $X(i)$ has effective homology, we present an algorithm computing the homotopy colimit $\hocolim X$ as a simplicial set with effective homology. We also give an algorithm computing the cofibrant replacement $X^\cofib$ of $X$ as a diagram with effective homology. This is applied to computing of equivariant cohomology operations.
}

\section{Introduction}
In recent years a framework of  \emph{simplicial sets with effective homology} developed by Sergeraert et al. (see e.g \cite{serg}) was used by a group of authors in papers  \cite{aslep, polypost, cmk, fv} to address a variety of computational problems in homotopy theory and algebraic topology. The main property of simplicial sets with effective homology is that we are able to perform homological computations with them. This is important when working with infinite simplicial sets such as the Eilenberg--MacLane spaces $K(\pi, n)$.

The framework also provides a collection of algorithms on simplicial sets with effective homology in such a way that if the inputs are simplicial sets with effective homology, the output are also simplicial sets with effective homology. These algorithms usually describe commonly used constructions from algebraic topology such as Cartesian product, loop space, bar construction, mapping cylinder, total space of fibration (see \cite{serg, cmk, polypost}) and homotopy pushout \cite{aslep, heras}.

In this paper, we present two main results. Firstly, we enlarge the above mentioned collection of constructions by presenting an algorithm that computes the Bousfield--Kan model of homotopy colimit of a finite diagram of simplicial sets with effective homology. We formulate this in the following  proposition.
\begin{prop}\label{prop:main}
Let $\icat$  be a finite category and $X\colon \icat \to \sSet$ a functor. Suppose that every $X(i)$ has effective homology for all $i\in \icat$ and that all maps in the diagram $X$ are computable. Then the space $\hocolim X$ has effective homology.
\end{prop}
 
 Secondly, we define the notion of a diagram of simplicial sets with effective homology. This notion is stronger than the notion of pointwise effective homology of a diagram defined by the requirement that every simplicial set in the diagram has effective homology. We consider the projective model structure on the category of functors $\icat \to \sSet$  \cite{hovey} in which weak equivalences and fibrations are pointwise weak equivalences and fibrations, respectively. Our second result reads as 
 
\begin{prop}\label{prop:submain}
Let $\icat$  be a finite category and $X\colon \icat \to \sSet$ a functor. Suppose that every $X(i)$ has effective homology for all $i\in \icat$ and that all maps in the diagram $X$ are computable. Then there is an algorithm which provides a diagram $X^{\cofib}$ which is a cofibrant replacement of $X$ and has effective homology as a diagram.
\end{prop}

In \cite{aslep}, the authors describe algorithmic approach for computing homotopy classes between spaces $X$ and $Y$ with a free action of a finite group if the dimension of $X$ is twice the connectivity of $Y$.
The main motivation for Proposition \ref{prop:submain} is the extension of this result to simplicial sets with nonfree action of a finite group using the fact the category $\sSet ^G$ of simplicial sets with an action of a group $G$ is Quillen equivalent to a category of functors $\Ocat_G ^\op \to \sSet$ where $\Ocat_G$ is a category of orbits of the group $G$.

In section \ref{sec:apl}, we present a special case of such a computation by computing the  equivariant cohomology operations, i.e. sets of equivariant homotopy classes $[K(\pi,n), K(\rho,m)]_G$ where $\pi$ and $\rho$ are diagrams of abelian groups \cite{bredon, dwykan, alaska} equipped with actions of a finite group $G$. These can be computed as the cohomology groups of the cofibrant replacement of $K(\pi,n)$ on the level of diagrams.

\section{Effective homology} \label{sec:ef}

We first repeat basic notions of effective homological algebra.

\subsection*{Locally effective simplicial sets}

The basic building blocks of the theory are locally effective simplicial sets and computable maps between them.

Let $X$ be a simplicial set. We say that $X$ is \emph{locally effective} 
if we are given a specified encoding of the simplices of $X$ and a collection of algorithms computing the face and degeneracy operations on the simplices of $X$. The notation suggests that we do not have any global information about $X$.
Similarly, a map of locally effective simplicial sets is \emph{computable} if there is an algorithm that for any given simplex in the domain gives the encoding for its image.

\subsection*{Effective chain complexes, reductions and strong equivalences}

Let $C_*$ be a chain complex. We call $C_*$ \emph{cellular} if there is an indexing set $A$ such that for every $\alpha \in A$ we are given $c_\alpha \in C_*$ such that
every chain $c \in C_*$ can be expressed uniquely as a combination
\begin{equation}
c=\sum k_{\alpha} c_{\alpha}\label{eq:cellular}
\end{equation}
with integer coefficients $k_\alpha$ in $\mathbb{Z}$.  In other words $C_*$ is cellular if we have a basis for $C_*$ such that any element $c\in C_*$ can be expressed uniquely as a combination of elements of this basis.

We call $C_*$ \emph{locally effective} if the elements of $C_*$ have finite encoding and there are algorithms computing the multiplication, addition, zero, inverse and differential for the elements of $C_*$. Further there is a basis and an algorithm that for every $c\in C_*$ computes \eqref{eq:cellular} and decides whether $c$ lies in the basis.

The chain complex $C_*$ is called \emph{effective} if it is locally effective and there is an algorithm that for given $n \in \mathbb{N}$ generates a finite basis $c_\alpha \in C_n$.

Clearly, the effective chain complexes have nice computational properties e.g. one can compute the homology of such a chain complex.

\begin{Def}
Let $C_*, C'_*$ be chain complexes. A \emph{reduction} $C_*\Ra C'_*$
is a triple $(\alpha,\beta,\eta)$ pictured as belowwhere $\alpha, \beta$ are chain maps and $\eta$ is a morphism of graded groups.
\[(\alpha,\beta,\eta)\col C_*\Ra C'_*\quad\equiv\quad\xymatrix@C=30pt{
C_* \ar@(ul,dl)[]_{\eta} \ar@/^/[r]^\alpha & C'_* \ar@/^/[l]^{\beta}
}\]
satisfying 
\begin{equation}\label{eq:reduction}
\begin{array}{lll}
 \eta\beta = 0 & \alpha \eta = 0& \alpha \beta  = \id \\
 \eta\eta = 0& \bo\eta + \eta \bo = \id - \beta\alpha&
\end{array}
\end{equation}

A \emph{strong equivalence} between chain complexes $C_*\LRa C'_*$ consists of a span of reduction
\[\xymatrix@C=20pt{ & \widehat C_* \ar@{=>}[dl] \ar@{=>}[dr] & \\
C_* & & C'_*
}\]
\end{Def}

Strong equivalences can be composed \cite{serg}, producing another strong equivalence. We say that a locally effective chain complex $C_*$ is equipped with \emph{effective homology} if there is a strong equivalence $C_*\LRa C_*^\ef$ of $C_*$ with some effective chain complex $C_*^\ef$. A locally effective simplicial set $X$ is said to have \emph{effective homology} if its chain complex $C_* (X)$ does.

Objects with effective homology further have nice properties:
\begin{lem} \label{lem:results}
Let $C_* , C'_*$ be chain complexes with effective homology and let $X, Y$ be simplicial sets with effective homology. Then the following holds:
\begin{enumerate}
\item $C_{*} \oplus C'_{*}$, $C_{*} \otimes C'_{*}$ have effective homology. 
\item The space $X \times Y$ has effective homology.
\end{enumerate}
\end{lem}
\begin{proof}
The proof that $C_{*} \oplus C'_{*}$ have effective homology is easy, for the proof that $C_{*} \otimes C'_{*}$ has effective homology, we refer the reader to \cite{serg}, Proposition 61.

For the second part, we use the Eilenberg--Zilber reduction $C_* (X \times Y) \Ra C_* (X) \otimes C_* (Y)$ first introduced in \cite{eml1, eml2}. The proof is finished using the first part of the statement.
\end{proof}

\subsection*{Perturbation Lemmas}
Consider a reduction $C_* \Ra D_*$. If we change the differential of one of the complexes, the following perturbation lemmas provide us with a new reduction $C'_* \Ra D'_*$ where the $C'_*$, $D'_*$  are the original chain complexes with changed (perturbed) differential.

\begin{Def}
Let $(C_*, \bo)$ be a chain complex. We call a collection of maps $\delta \colon C_* \to C_{*-1}$ \emph{perturbation} if the sum $\bo + \delta$ is also a differential on $C_*$.
\end{Def}

The following perturbation lemmas are well--known, tehy constitute one of the basic tools in homological perturbation theory. Their genesis can be traced back to \cite{eml1, brown, shih, gug} and their full proofs can be found e.g. in \cite{serg}. 

\begin{lem}[Easy Perturbation Lemma]
Let $(\alpha,\beta,\eta) \col (C_*, \bo)\Ra  (C'_*,\bo')$ be a reduction. Suppose $\delta'$ is a perturbation of differential $\bo'$. Then there is a reduction $(\alpha,\beta,\eta) \col (C, \bo + \beta \delta' \alpha) \Ra  (C',\bo' + \delta)$.
\end{lem}
\begin{proof}
Check the formulas for the new reduction given in the statement.
\end{proof}

\begin{lem}[Basic Perturbation Lemma]
Let $(\alpha,\beta,\eta) \col (C_*, \bo)\Ra  (C_*,\bo')$ be a reduction. Suppose $\delta$ is a perturbation of differential $\bo$ such that let for every $c \in C_*$ there is some $i \in \mathbb{N}$ satisfying $(\eta\delta)^{i} (c) = 0$. Then there is a reduction 
 $(\alpha',\beta',\eta') \col (C, \bo + \delta) \Ra  (C,\bo' + \delta')$.
\end{lem}
\begin{proof}
We describe the formulas for maps in the new reduction and the rest is left for the reader. We set
\[
\varphi = \sum_{i =0} ^{\infty} (-1)^i (\eta\delta)^{i}; \qquad \psi = \sum_{i =0} ^{\infty} (-1)^i (\delta\eta)^{i}
\]
By the condition in the statement, both sums are finite. The maps in the reduction are given as follows:
\[
\begin{array}{llll}
\delta' = \alpha \psi \delta \beta = \alpha  \delta \phi \beta; & \alpha' = \alpha \psi; & \beta' = \phi \beta;& \eta' = \phi \eta = \eta \psi.
\end{array} \qedhere
\]
\end{proof}

\subsection*{Effective homology for diagrams}
\begin{Def}
We call a diagram $C \colon\icat \to \Ch_+$  \emph{pointwise locally effective} if $C (i)$ is locally effective for every $i \in \icat$. If further for every every morphism $f$ in the category $\icat$ is the morhism $C(f)$ computable,  we call $C$ \emph{locally effective}.

We say that pointwise locally effective $C$ has \emph{pointwise effective homology} if for every $i\in \icat$ there is an effective chain complex $C ^{\ef} (i) $ and a strong equivalence of chain complexes $C (i) \LRa C  ^{\ef} (i)$.
\end{Def}

\begin{Def}
Let $C\colon \icat \to \Ch_{+}$. We say $C$ is \emph{cellular} if there is an indexing set $A$ such that for every $\alpha \in A$ we are given 

\begin{itemize}
\item $i_\alpha \in \icat$,
\item $c_\alpha $ is a chain in $C(i_\alpha)$,
\end{itemize}
such that the set
\[
\{f_* c_\alpha \mid \alpha \in A, f \in  \icat(i_\alpha, i) \}
\]
forms a basis for each $C(i)$.
\end{Def}

We can formulate the cellularity also in a different way: given an element $c \in C (i)$ there is a unique description of $c$ as 
\begin{equation}\label{eq:cell}
c= \sum\limits_{\alpha} {f_\alpha}_* (c_\alpha)
\end{equation}
where ${f_\alpha}_*$ is an element in $\mathbb{Z}\icat(i_\alpha, i)$ the free abelian group on the set $\icat(i_\alpha, i)$.

\begin{Def}
We call a locally effective diagram $C$ \emph{effective} if there is an algorithm that generates for given $n$ finite list of all elements $c_\alpha \in C (i_\alpha)_n$ and an algorithm computing \eqref{eq:cell} for this basis.
\end{Def}

We remark that the notion of a cellular functor $C$ is somewhat related to the notion of a functor that is free or representable  with models $\mathsf{ob}(\ccat)$ one can find more details e.g. in \cite[pp. 127]{may}. 
\begin{ex}

Let $\icat = G^{\op}$, where $G$ is a finite group thought of as a category with one object and arrows labelled by the elements of $G$. Let $C:\icat \to \Ch_{+}$. Then the cellularity means that every element $c\in C(i)$ can be described uniquely as 
\[
c= \sum\limits_{\alpha}  c_\alpha k_\alpha
\]
where $k_\alpha \in \mathbb{Z}G$. Hence, the chain complex is cellular only if $G$ acts freely.

\end{ex}

\begin{ex} \label{ex:main}
Let $\icat$ be a finite category. Then for any $i\in \icat$ there is a functor $\mathbb{Z}\icat(i,-) \colon \icat \to \mathsf{Ab}$ the free abelian group on the set $\icat(i,-)$. We think of this abelian group as a chain complex concentrated in degree $0$ and thus obtain a functor $\mathbb{Z}\icat(i,-) \colon \icat \to \Ch_+$. This diagram of chain complexes is effective.

We first show it is cellular:
The basis consists of one element only, namely $\id_i$. Given some $j \in \icat$, the elements $x \in \icat(i,j)$ form the basis $\mathbb{Z}\icat(i,-)$ and we can describe them as $x=  {x}_*(\id_i) $. The finiteness of $\icat$ now implies that $\mathbb{Z}\icat(i,-)$ is effective.
\end{ex}

To define diagrams with effective homology, we introduce reduction and strong equivalence of diagrams:

\begin{Def}
Let $C,C' \colon \icat \to \Ch_+$ be diagrams of chain complexes. A \emph{reduction} $C \Ra C'$  is a triple of natural transformations  $(\alpha,\beta,\eta)$. 

\[(\alpha,\beta,\eta)\col C\Ra C'\quad\equiv\quad\xymatrix@C=30pt{
C \ar@(ul,dl)[]_{\eta} \ar@/^/[r]^\alpha & C' \ar@/^/[l]^{\beta}
}\]
they again satisfy the conditions \eqref{eq:reduction}.

Similarly we define strong equivalence of  diagrams of chain complexes and the strong equivalences can again be composable.
We say that a locally effective diagram $C\colon \icat  \to \Ch_+$ has \emph{effective homology} if there is a strong equivalence of locally effective diagrams $C \La \widehat{C} \Ra C  ^\ef$ where $C ^\ef\colon \icat \to \Ch_+$ is an effective diagram of chain complexes. 
A diagram of simplicial sets $X \colon \icat \to \sSet$ has \emph{effective homology} if $C(X)$ has effective homology.
\end{Def}

\subsection*{Perturbation lemmas for diagrams}
We now define perturbation and give perturbation lemmas for diagrams of chain complexes:

\begin{Def}
Let $C,C'\colon \icat \to \Ch_+$. Notice that the differential $\bo$ on $C$ can be seen as a natural transformation $C \to C [1]$ satisfying $\bo\bo = 0$. Here $C[1]$ is diagram of chain complexes $C$ with all the chain complexes moved up by one dimension.
We call a collection of maps $\delta : C \to C[1] $ \emph{perturbation} if the sum $\bo + \delta$ is also a differential.
\end{Def}

We now formulate the lemmas.
\begin{lem}[Easy Perturbation Lemma] \label{lem:epl}
Let $(\alpha,\beta,\eta) \col (C, \bo)\Ra  (C',\bo')$ be a reduction of diagrams of chain complexes . Suppose $\delta'$ is a perturbation of differential $\bo'$. Then there is a reduction $(\alpha,\beta,\eta) \col (C, \bo + \beta \delta' \alpha) \Ra  (C',\bo' + \delta)$.
\end{lem}
\begin{lem}[Basic Perturbation Lemma] \label{lem:bpl}
Let $(\alpha,\beta,\eta) \col (C, \bo)\Ra  (C,\bo')$ be a reduction of diagrams of chain complexes. Suppose $\delta$ is a perturbation of differential $\bo$ and further  for every $i\in \icat$ and every $c \in C (i)$ there is some $k \in \mathbb{N}$ such that we get  $(\eta\delta)^{k} (c) = 0$. Then there is a reduction of diagrams of chain complexes
 $(\alpha',\beta',\eta') \col (C, \bo + \delta) \Ra  (C' ,\bo' + \delta')$.
\end{lem}
Both lemmas can be proven in the same way as classical perturbation lemmas.
 We have concrete description of the new reductions and they are given as sums of compositions of $\alpha, \beta, \eta, \bo, \delta$. Therefore it is enough to check the previously described formulas.

 \begin{lem} \label{prop:product}
 \
\begin{enumerate}
\item Let $X,Y \colon \icat \to \sSet$ be diagrams of simplicial sets with pointwise effective homology. Then the diagram $(X \times Y) \colon \icat \to \sSet$ has pointwise effective homology.
\item Let $X\colon \icat \to \sSet$,  $Y\colon \jcat \to \sSet$ be diagrams with effective homology. Then $X \mathbin{\widehat{\times}} Y \colon {\icat} \times {\jcat} \to \sSet$ has effective homology.
\item Let $C \colon \icat \to \Ch_+$,  $C'\colon \icat \to \Ch_+$ be diagrams of chain complexes with effective homology. Then $C \oplus C'$ has effective homology.
\end{enumerate}
\end{lem}

\begin{proof}
It is enough to show that for each $i \in \icat$, $C((X \times Y) (i)) $ has effective homology and this is proven by the second part of Lemma \ref{lem:results}.

In the second part we use the functoriality of Eilenberg--Zilber reduction (see \cite{eml2}, Theorem 2.1a). The diagram $C (X \times Y)$ is strongly equivalent to diagram $C^\ef (X) \otimes C ^\ef (Y)$ and it remains to show the latter is effective.

Let $x_\alpha$ be the finite basis for $C^\ef (X) $ and $y_\beta$ be the basis of $C ^\ef (Y)$. It is well-known that the basis of tensor product is formed by tensor products of basis elements, so the basis of $C^\ef (X) \otimes C ^\ef (Y)$ is generated by the set
\[
\{f_* x_\alpha \otimes g_* y_{\beta}\mid f \in  \icat(i_\alpha, i), g\in \jcat(j_{\beta, i}) \}  = 
\{(f , g)_* x_\alpha \otimes  y_{\beta}\mid(f, g) \in  {\icat} \mathbin{\widehat{\times}} {\jcat} ((i_\alpha, j_{\beta}),  (i, j)) \} 
\]
The last part is trivial.
\end{proof}

\begin{cor}\label{cor:otimes}
Let $C \colon \icat \to \Ch_+$ be a diagram with effective homology and let $C'$ be a chain complex with effective homology. Then $C' \otimes C \colon \icat \to \Ch_+$ has effective homology.
\end{cor}
\begin{proof}
We can see $C'$ as a diagram $C' \colon * \to \Ch_+$. The diagram $C' \otimes C$ is strongly equivalent to some  ${C'} ^\ef \otimes {C} ^\ef \colon \icat \to \Ch_+$. Lemma \ref{prop:product} (2) then gives the result.
\end{proof}

In what follows we are going to apply a general lemma about filtered diagrams of chain complexes.
Let  $C \colon \icat \to Ch_+$. We introduce a filtration $F$ on diagram $C$ of chain complexes: 
\[
0 = F_{-1} C \subseteq F_0 C  \subseteq F_1 C  \subseteq \cdots
\]
such that $C = \bigcup_k F_k C$. We further assume that each $F_k C$ is a cellular subcomplex i.e. it is generated by a subset of the given basis of $C$ and that the filtration is locally finite i.e. for each $n$ we have $F_k C_n = C_n$ for $k \gg 0$.

\begin{lem}[\cite{aslep}, Lemma 7.3] \label{lem:main}
Let $C$ be a diagram of chain complexes with filtration $F$ satisfying properties as above. If each filtration quotient $G_k C = F_k C / F_{k-1} C$ is a diagram with effective homology  then so is $C$.
\end{lem}
\begin{proof}
We define $G = \bigoplus\limits_{k \geq 0} G_k$. The sum is not finite, but it is \emph{locally finite}: By the properties of $F$, we get that $G_k C_n = 0$ for $k \gg 0$. Thus for each $n$, we get a direct sum of diagrams with effective homology $G_k C_n \colon \icat \to \Ch_+$ and it follows that $G$ has effective homology. 

The diagram $C$ differs form $G$ only by a perturbation of its differential. This perturbation decreases the filtration degree. If we take a direct sum of given strong equivalences $G_k \La \widehat{G_k}\Ra G_k ^\ef$, we obtain a strong equivalence $G  \La \widehat{G}\Ra G ^\ef$. All the chain complexes are equipped with a filtration degree. Since the perturbation on $G$ decreases the filtration degree, while the homotopy operator preserves it, we can apply the perturbation lemmas \ref{lem:epl}, \ref{lem:bpl} to obtain a strong equivalence $C  \La \widehat{C}\Ra C ^\ef$. \end{proof}
\section{Proofs of main results}

We remark, that both homotopy colimit $\hocolim X$ and cofibrant replacement $X ^\cofib$ of a diagram $X \colon \icat \to \sSet$ can be seen as cases of \emph{homotopy left Kan extension} $\holan_p X$ \cite{bous, hirsch}, where $p$ is a functor $\icat \to \jcat$ . We picture the situation in the following diagram:
\[
\xymatrix{
{\icat} \ar[rr]^{X} \ar[dr]_p & & \sSet \\
{ } & \jcat \ar@{.>}[ur]_{\holan_p X}& {}
}
\]
For the homotopy colimit one chooses $\jcat = *$ and $p$ the unique functor to the terminal category and for the cofibrant replacement we set $\jcat = \icat$ and $p = \id$ \cite{bous, alaska}.
We will make use of the following Bousfield--Kan formula for $\holan_p X$ (see \cite{bous, isac, dugger}):
\begin{equation}\label{eq:holan}
(\holan_p X) (j) = \bigsqcup_{n} \bigsqcup_{i_0, \cdots , i_n} \stdsimp{n} \times X(i_0) \times \icat (i_0, i_1) \times \cdots \times \icat (i_{n-1}, i_n) \times \jcat( p(i_n), j ) / {\sim}
\end{equation}
here the relation $\sim$ is given as 

\begin{align*}
(d^k t, x, f_1, f_2, \ldots f_n, g) \sim {} & (t, x, f_1, f_2, \ldots ,f_{k+1}  f_{k}  , \ldots, f_{n-1}, f_n, g),  & & 0< k< n; \\
(d^k t, x, f_1, f_2, \ldots f_{n}, g) \sim {} & (t, x, f_1, f_2, \ldots, f_{n-2}, f_{n-1}, g p(f_n)), & & k= {n};\\
(d^k t, x, f_1, f_2, \ldots f_{n}, g)\sim {} & (t, f_1(x), f_2,  \ldots , f_{n-1}, f_{n}, g),  &  &k= 0.\\
\end{align*}

For the choices of $\jcat, p$ we get corresponding formulae for homotopy colimit and cofibrant replacement namely

\[
\hocolim X = \bigsqcup_{n} \bigsqcup_{i_0, \cdots , i_n} \stdsimp{n} \times X(i_0) \times \icat (i_0, i_1) \times \cdots \times \icat (i_{n-1}, i_n)/ {\sim} 
\]

\[
X^{\cofib} (-) = \bigsqcup_{n} \bigsqcup_{i_0, \cdots , i_n} \stdsimp{n} \times X(i_0) \times \icat (i_0, i_1) \times \cdots \times \icat (i_{n-1}, i_n) \times \icat (i_n,  - ) / {\sim} 
\]
with the obvious specified relations.

We now formulate theorem regarding effective homology of the Bousfield--Kan model of $\holan_p X$ and we obtain both Proposition \ref{prop:main} and \ref{prop:submain} as straightforward corollaries.

\begin{thm}\label{thm:main}
Let $X \colon \icat \to \sSet$ be a pointwise effective diagram, $p\colon \icat \to \jcat$ a functor between finite categories. Then $\holan_p X \colon \jcat \to \sSet$ is diagram with effective homology.
\end{thm}
\begin{proof}
For any category $\icat$ there is a simplicial set $N \icat$, the nerve of $\icat$. The simplicial set $N \icat$ can be seen as a homotopy colimit of the diagram consisting of points. Then there is a projection  $q \colon \holan_p X \to N \icat $ given as a projection onto
\[
\bigsqcup_{n} \bigsqcup_{i_0, \cdots , i_n} \stdsimp{n} \times \icat (i_0, i_1) \times \cdots \times \icat (i_{n-1}, i_n)/ {\sim}
\]
and we define the skeleton of $\holan_p X$:
\[
\sk_{k} \holan_p X = q^{-1} (sk_{k} N \icat).
\]

Let $C \colon \sSet \to \Ch_+$ denote the standard chain complex functor. We want to use Lemma \ref{lem:main} prove that the diagram $C(\holan_p X) \colon \icat \to \Ch_+$ has effective homology. Therefore, we first have to introduce a filtration $F$ on the diagram of chain complexes $C(\holan_p X)$. We define $F$ as follows:
\[
F_k C (\holan_p X ) = C ( \sk_{k} \holan_p X)
\]
Denoting $G_k = F_k / F_{k-1}$, we get
\[
G_k (C (\holan_p X)) = \bigoplus\limits_{\substack{i_0\to \cdots \to i_k \\ \text{ nondeg.} }}  C (\stdsimp{k} \times X(i_0)  \times \jcat(p(i_k), -) , \bo \stdsimp{k} \times X(i_0)  \times \jcat(p(i_k), -)) . \\
\]
The sum is taken over chains of morphisms in $\icat$ that do not contain identity as those are canceled out when computing $G_k = F_k/ F_{k-1}$. By the finiteness of $\icat$, for each $k$ there number nondegenerate chains of morphisms of length $k$ is finite, so the sum is finite.

Using the Eilenberg--Zilber reduction we get that $G_k (C (\holan_p X))$ is strongly equivalent to 
\begin{equation} \label{eq:formula}
 \bigoplus\limits_{\substack{i_0\to \cdots \to i_k \\ \text{ nondeg.}}}  C (\stdsimp{k}, \bo \stdsimp{k}) \otimes C(X(i_0))  \otimes \mathbb{Z}\jcat (p(i_k), -) \cong  \bigoplus\limits_{\substack{i_0\to \cdots \to i_k \\ \text {nondeg.}}} s^kC(X(i_0))  \otimes \mathbb{Z}\jcat (p(i_k), -)
\end{equation}
where $s$ denotes the suspension. To finish the proof, it remains to show that under the assumptions of Theorem \ref{thm:main} the diagrams $G_k$ have effective homology.

In Example \ref{ex:main}, we have shown that $\mathbb{Z}\jcat (p(i_k), -)$ is effective diagram of chain complexes. Therefore it has effective homology. Further $s^k C(X(i_0))$ is chain complex with effective homology. Using Corollary \ref{cor:otimes} we get that $s^k C(X(i_0))  \otimes \mathbb{Z}\jcat (p(i_k), -)$ is diagram of chain complexes with effective homology. As $G_k (C (\holan_p X))$ is strongly equivalent to a finite direct sum of chain complexes with effective homology, it has effective homology. Now we can apply Lemma \ref{lem:main} to complete the proof.
\end{proof}

\section{Application} \label{sec:apl}
In paper \cite{cmk} the authors presented algorithm computing the set of homotopy classes of maps $[X,Y]$\footnote{We remark that all through this section, we are using the standard notation from  model categories, so by $[X,Y]$ we in fact mean $[X^\cofib, Y^\mathsf{fib}]$.} where $X,Y$ are simplicial sets satisfying certain conditions. One of the steps in their construction based on building a Postnikov tower for $Y$ was to compute $[X, K(\pi, n)]$ where $X$ is a simplicial set with effective homology and $K(\pi, n)$ is a simplicial model for the Eilenberg--MacLane space with \emph{fully effective abelian} group $\pi$. In paper \cite{aslep} this result was extended to computing the homotopy classes of equivariant maps $[X,Y]_G$ where $G$ is a finite group with free actions on $X$ and $Y$. Their construction also involved computing $[X, K(\pi, n)]_G$ where $\pi$ is an abelian group with an action of $G$.

In this section we generalize the computation of  $[X, K(\pi, n)]$ for $X$ and  $K(\pi, n)$ being diagrams of simplicial sets. As a corollary we will further get an algorithm computing the cohomology operations  $[K_G(\pi, n), K_G(\rho, k)]_G$ where $G$ is a finite group that acts on Eilenberg--MacLane $G$-simplicial sets $K_G(\pi, n), K_G(\rho, k)$. this can be considered as a generalization of some constructions in \cite{eml1} and \cite{eml2}

We again assume the projective model structure on the category of functors $\icat \to \sSet$. Given a diagram of Abelian groups $\pi \colon \icat \to \Ab$ and a diagram of simplicial sets $X \colon \icat \to \sSet$, we define the cochain complex 
$C^* (X;\pi) = \Hom(C_* X, \pi)$. A group $H^n(X; \pi)$ the $n$th \emph{cohomology group} of $X$ with coefficients in $\pi$ is then given by the homology of $C^* (X;\pi)$.

Further, $\pi$ induces an \emph{Eilenberg--MacLane object} $K(\pi, n)$ (see \cite{dwykan}), which is in fact a diagram of Eilenberg--MacLane spaces $K(\pi (i), n)$. If not stated otherwise, by $K(\pi, n)$ we will always mean the following model which is due to \cite[Theorem 23.9]{may}:
\[
K(\pi, n)_q = Z^n (\stdsimp{q}, \pi) .
\] 
where $\stdsimp{q} \in \sSet$ is seen as a trivial diagram $\icat \to \sSet$. The following proposition uses the notion of fully effective abelian group. We give the proper definition later, for now we only remark that a fully effective abelian group $A$ is a computable description of all generators and relations of $A$.

\begin{prop} \label{prop:apl}
Let $\icat$  be a finite category and $X\colon \icat \to \sSet$ a functor. Suppose that  $X(i)$ has effective homology for all $i\in \icat$ and that all maps in the diagram $X$ are computable. Let $\pi\colon \icat \to \Ab$ be a diagram of fully effective abelian groups. Then there is an algorithm which computes $[X, K(\pi, n)]$.
\end{prop}
We postpone the proof until later and formulate a corollary regarding $G$-simplicial sets.
Let $G$ be a finite abelian group and let $X$ be a simplicial set with a $G$ action. Considerthe category $\GsSet$ of $G$-simplicial sets. This category can be equipped with a model structure, called \emph{fixed point} model structure, see \cite{alaska}.

Further, there is a category of orbits $\Ocat_G$, where the objects are the orbits $G/H$ for $H  \leq G$ and the morphisms are equivariant maps $G/H_1 \to G/H_2$. By Elemendorf's theorem \cite{elmendorf} the categories $\GsSet$ with fixed point model structure and  $\OgsSet$ with projective model structure are Quillen equivalent. We remark, that the right Quillen fully faithful functor $\Phi \colon \GsSet \to \OgsSet$ is defined by $\Phi (X) (G/H) = X^H = \{x\in X \mid hx = x, \forall h\in H\}$. Using this functor, we define the Bredon cohomology $H^* _G ( X; \pi)$ of some $X\in \GsSet$ with coefficients in a diagram $\rho\colon \Ocat_G ^\op \to \Ab$ as $H^* _G (X; \rho) = H^* (\Phi X; \rho)$.

Using the left Quillen functor $\Psi  \colon\OgsSet  \to \GsSet$ given by $\Psi (T) = (T^\cofib)(G/e)$, we define the Eilenberg--MacLane $G$-simplicial set $K_G (\rho, n) = \Psi K(\rho, n)$. Now we can state:
\begin{cor}
Let $G$ be a finite group and $\pi, \rho \colon \Ocat_G ^\op \to Ab$ diagrams of fully effective abelian groups. Then there is an algorithm computing $[K_G(\pi, n), K_G(\rho, k)]_G$.
\end{cor}
\begin{proof}
Using the Elmendorf's theorem \cite{elmendorf, stephan, alaska}, we get that  
\[
[K_G(\pi, n), K_G(\rho, k)] \cong  [K(\pi, n), K(\rho, k)]
\]
where $K(\pi, n)$ is a diagram with pointwise effective homology \cite{krcal, polypost}. The proof is finished using Proposition \ref{prop:apl}. Notice that we have computed $H^k _G (K_G ^\cofib (\pi, n); \rho)$.
\end{proof}

We now sum up the necessary information about fully effective abelian groups. For more detailed information we refer the reader to \cite{cmk}.
\subsection*{Fully effective abelian groups}

A \emph{fully effective abelian} group $A$ consists of 
\begin{itemize}
\item a set of representatives $\mathcal{A}$ which we imagine are stored in a computer. The element represented by an $\alpha \in \mathcal{A}$ is denoted $[\alpha]$,
\item algorithms that provide us with a representative for neutral element, product and inverse. In more detail we can compute $0\in \mathcal{A}$ such that $[0] = e$, given any $\alpha, \beta \in \mathcal{A}$ we compute $\gamma \in \mathcal{A}$ such that $[\gamma] = [\alpha] + [\beta]$ and for any $\alpha \in \mcA$ we can compute $\beta \in \mcA$ such that $[\beta] = - [\alpha]$,
\item a list of generators $a_1, \ldots, a_r$ (given by by the representatives) and numbers $q_1,  \ldots, q_r \in \{ 2,3,, \ldots \} \cup \{ 0\}$,
\item An algorithm that given $\alpha \in \mathcal{A}$ computes integers $z_1, \ldots, z_r$ such that $[\alpha] = \sum_{i=1} ^r z_i a_i$.
\end{itemize}

We call a mapping $f \colon A \to B$ of fully effective abelian groups  \emph{computable homomorphism} if there is a computable mapping of sets $\phi \colon \mathcal{A} \to \mathcal{B}$ such that $f([\alpha] = [\phi(\alpha)])$.

In other words, the above mentioned structure of fully effective abelian group $A$ gives us a computable isomorphism (together with its inverse) $A \cong \bbZ/q_1 \oplus \cdots \oplus \bbZ/q_r$ where $\bbZ/0 \cong \bbZ$.

The proof of the following lemma can be found in \cite{cmk}.
\begin{lem}\label{lem:ker}
Let $f \colon A \to B$ be a computable homomorphism of fully effective abelian groups. Then both $\Ker (f)$ and $\coker (f)$ can be represented as fully effective abelian groups.
\end{lem}

Let $\icat$ be a finite category and let $\pi \colon \icat \to \Ab$ be a diagram such that every $\pi(i)$ is fully effective abelian and every morphism is computable homomorphism. We will call such diagram a \emph{diagram of fully effective abelian groups}. As a consequence of the previous lemma, we get
\begin{lem}\label{lem:hom}
Let $\icat$ be a finite category and let $\pi, \rho \colon \icat \to \Ab$ be a diagrams of fully effective abelian groups. Then $\Hom(\pi, \rho)$ is a fully effective abelian group.
\end{lem}
\begin{proof}
First notice that each $\Hom(\pi(i), \rho (i')) $ is a fully effective abelian group. Secondly, we can see that $\Hom(\pi, \rho) \leq \prod_{i\in \icat} \Hom(\pi(i), \rho (i))$. We define a homomorphism
\[
F\colon  \prod_{i\in \icat} \Hom(\pi(i), \rho (i)) \to  \prod_{f\colon i\to i'}\Hom(\pi(i), \rho (i'))
\]
for any $g \in \prod_{i\in \icat} \Hom(\pi(i), \rho (i)) $ as follows:
\[
F(g) =   (\rho(f)g(i) - g(i')\pi(f))_{f\colon i\to i'}.
\]
Then the desired $\Hom(\pi, \rho)$ is equal to $\Ker F$. As $\icat$ is a finite category , $ \prod_{i\in \icat} \Hom(\pi(i), \rho (i))$,  $\prod_{f\colon i\to i'}\Hom(\pi(i), \rho (i'))$ are both fully effective abelian groups. Because $F$ is computable, Lemma \ref{lem:ker} gives us that $\Ker F$ is fully effective.
\end{proof}
\subsection*{Proof of Proposition \ref{prop:apl}}
We remark that $K(\pi, n)$ is already a fibrant object, and by Proposition \ref{prop:submain}, $X^\cofib$ has effective homology. We must now compute  $[X^\cofib, K(\pi, n)]$ Thanks to the model we are using, we can use the constructions in \cite{may}, section 24, as presented in \cite{goja}, p. 345, and we get an isomorphism $H^n(X^\cofib; \pi) \cong [X^\cofib, K(\pi, n)]$. By Proposition \ref{prop:submain}, the chain complex $C_* X^\cofib$ has effective homology and this in particular implies that each $C_n  X^\cofib$ is a diagram of fully effective abelian groups.  By Lemma \ref{lem:hom} each $ \Hom(C_k X^\cofib, \pi)$ is a fully effective abelian group and using Lemma \ref{lem:ker}, we get that  $H^n \Hom(C_* X^\cofib, \pi)$ is an effective abelian group, which concludes the proof.\qed

\subsection*{Acknowledgements}
The author would like to thank L.Vok\v{r}\'{i}nek and for his helpful advice and remarks in particular with the formulation of Proposition \ref{prop:main}. I also thank my supervisor M.\v{C}adek for useful discussions, suggestions and comments.

\end{document}